\documentclass[11pt]{amsart} %%% {amsproc}

\usepackage{a4wide,amssymb,color,textcomp}
\usepackage{mathrsfs}
\usepackage[all]{xy}

\makeatletter
\@namedef{subjclassname@2020}{%
  \textup{2020} Mathematics Subject Classification}
\makeatother

\makeatletter
\DeclareFontEncoding{LS1}{}{}
\DeclareFontSubstitution{LS1}{stix}{m}{n}
\DeclareMathAlphabet{\mathcal}{LS1}{stixscr}{m}{n}
\makeatother

\newcommand{\A}{\mathscr A}
\newcommand{\F}{\mathcal F}
\newcommand{\I}{\mathcal I}
\newcommand{\J}{\mathcal J}
\newcommand{\U}{\mathcal U}
\newcommand{\C}{\mathcal C}
\newcommand{\V}{\mathcal V}

\newcommand{\M}{\mathcal M}
\newcommand{\N}{\mathcal N}
\newcommand{\K}{\mathcal K}
\newcommand{\w}{\omega}
\newcommand{\IN}{\mathbb N}
\newcommand{\IR}{\mathbb R}
\newcommand{\Ra}{\Rightarrow}
\newcommand{\IQ}{\mathbb Q}
\newcommand{\Bo}{\mathcal{B\!o}}
\newcommand{\Ba}{\mathcal{B\!a}}
\newcommand{\BaI}{\mathcal{B\!a}^\pm\!\I}
\newcommand{\BoI}{\mathcal{B\!o}^\pm\!\I}
\newcommand{\AI}{\A^\curlyvee\!\I}
\newcommand{\defeq}{\overset{\mbox{\tiny\sf def}}=}
\newcommand{\Ker}{\mathrm{Ker}}

\newtheorem{theorem}{Theorem}[section]
\newtheorem{lemma}[theorem]{Lemma}

\newtheorem{corollary}[theorem]{Corollary}
\newtheorem{proposition}[theorem]{Proposition}
\newtheorem{problem}[theorem]{Problem}

\theoremstyle{definition}
\newtheorem{definition}[theorem]{Definition}
\newtheorem{remark}[theorem]{Remark}

\title[Automatic continuity of measurable homomorphisms on topological groups]{Automatic continuity of measurable homomorphisms on \v Cech-complete topological groups}
\author{Taras Banakh}

\address{T.Banakh: Ivan Franko National University of Lviv (Ukraine) and Jan Kochanowski University in Kielce (Poland)}
\email{t.o.banakh@gmail.com}



\begin{document}

\begin{abstract} We prove that a homomorphism $h:X\to Y$ from a (locally compact)  \v Cech-complete topological group $X$ to a topological group $Y$ is continuous if and only if $h$ is Borel-measurable if and only if $h$ is universally measurable (if and only if $h$ is Haar-measurable). This answers a problem of Kuznetsova and extends a result of Kleppner on the continuity of Haar-measurable homomorphisms between locally compact groups and a result of Rosendal on the continuity of universally measurable homomorphisms between Polish groups. 
\end{abstract}
\maketitle

\section{Introduction}

The problem of automatic continuity of homomorphisms between topological groups traces its history back to Cauchy \cite{Cauchy} who proved that each continuous additive function $f:\IR\to\IR$ is linear and posed the problem of existence of discontinuous additive functions on the real line. This problem was resolved in 1905 by Hamel who introduced the notion of a Hamel basis and using this notion constructed many (namely $2^{\mathfrak c}$) discontinuous additive functions on the real line. In 1913 Fr\'echet \cite{Frechet} proved  that every Lebesgue measurable additive function on the real line is continuous. In the first issue of Fundamenta Mathematicae three papers (of Banach \cite{Banach}, Sierpinski \cite{Sierpinski} and Steinhaus \cite{Steinhaus}) were devoted to the problem of automatic continuity of Lebesgue measurable additive functions on the real line. Those results were further developed in the framework of the  theory of functional equations and inequalities \cite{Kuczma}.

One of classical results of this theory (see \cite[Ch.5]{Chris}, \cite[9.10]{Ke}, \cite{Ros09}) says that a homomorphism $h:X\to Y$ between (locally compact) Polish groups is continuous if and only if $h$ is Borel-measurable (if and only if $h$ is Haar-measurable). A function $f:X\to Y$ from a locally compact topological group $X$ to a topological space $Y$ is {\em Haar-measurable} if for any open set $U\subseteq Y$ the preimage $f^{-1}[U]$ belongs to the $\sigma$-algebra  of measurable sets with respect to a Haar-measure on $X$ (which is known to be unique up to a multiplicative constant). % In \cite{Chris} Christensen proved that under the assumption $2^{\w_1}>2^{\w_0}$ every Borel-measurable homomorphism from an Abelian complete metric group to any topological group is continuous.

In \cite{Klep1}, \cite{Klep2} Kleppner proved that any Haar-measurable homomorphism between locally compact topological groups is continuous.    In \cite{Kuz} Kuznetsova applied  Martin's Axiom to show that every Haar-measurable homomorphism $h:X\to Y$ from a locally compact topological group $X$ to any topological group $Y$ is continuous, and asked whether Martin's Axiom can be removed from her result.  

This indeed can be done as shown by the following theorem, which is one of four  principal results of this paper.

\begin{theorem}\label{t:main-H} Every Haar-measurable homomorphism $h:X\to Y$ from a locally compact topological group $X$ to any topological group $Y$ is continuous.
\end{theorem}

Theorem~\ref{t:main-H} will be applied to prove the continuity of universally measurable homomorphisms on \v Cech-complete groups. 

A topological group $X$ is {\em \v Cech-complete} if its underlying topological space is \v Cech-complete, i.e. $X$ is homeomorphic to a $G_\delta$-subset of some compact Hausdorff space. It is well-known \cite[4.3.26]{Eng} that a metrizable space is \v Cech-complete if and only if its topology is generated by a complete metric. 

A function $f:X\to Y$ between topological spaces is {\em universally measurable} if for every open set $U\subseteq Y$ the preimage $f^{-1}[U]$ is measurable with respect to any probability Radon measure on $X$. 

In \cite{Ros19} Rosendal proved that every universally measurable homomorphism between Polish groups is continuous, thus resolving an old problem of Christensen.  In the following theorem we extend this result of Rosendal to all \v Cech-complete groups.

\begin{theorem}\label{t:main-UM} Every universally measurable homomorphism $h:X\to Y$ from a \v Cech-complete topological group $X$ to an arbitrary topological group $Y$ is continuous.
\end{theorem} 

Also we prove an analogous continuity criterion for (universally) BP-measurable homomorphisms on $\w$-narrow \v Cech-complete groups.

A topological group $X$ is {\em $\w$-narrow} if for every nonempty open set $U$ in $X$ there exists a countable set $C\subseteq X$ such that $X=CU=UC$. This important class of topological groups was introduced and studied by Guran \cite{Guran}, see also \cite[\S3.4]{AT}.

A function $f:X\to Y$ between topological spaces is called {\em BP-measurable} if for any open set $U\subseteq Y$ the preimage $f^{-1}[U]$ belongs to the $\sigma$-algebra of sets with the Baire Property in $X$. This $\sigma$-algebra is defined as the smallest $\sigma$-algebra that contains all open and all meager subsets of $X$.
A classical result of Pettis \cite{Pettis} implies that {\em any BP-measurable homomorphism $h:X\to Y$ from a Baire topological group to an $\w$-narrow topological group $Y$ is continuous}, see Theorem 2.2 in \cite{Ros09}. Our third principal result is the following automatic continuity criterion.

\begin{theorem}\label{t:main-BP} Every BP-measurable homomorphism $h:X\to Y$ from an $\w$-narrow \v Cech-complete topological group $X$ to any topological group $Y$ is continuous.
\end{theorem}

This theorem will be applied to show that every universally BP-measurable homomorphism on a \v Cech-complete group is continuous. A function $f:X\to Y$ between topological spaces is {\em universally BP-measurable} if for every open set $U\subseteq Y$ the preimage $f^{-1}[U]$ is {\em universally BP-measurable} in the sense that for every closed subset $F\subseteq X$ the set $F\cap f^{-1}[U]$ has the Baire Property in $F$.

\begin{theorem}\label{t:main-uBP} Every universlly BP-measurable homomorphism $h:X\to Y$ from a \v Cech-complete topological group $X$ to any topological group $Y$ is continuous.
\end{theorem}

\begin{proof} Since \v Cech-complete spaces are $k$-spaces \cite[3.9.5]{Eng}, the continuity of $h$ will follow as soon as we check that for every compact subset $K\subseteq X$ the restriction $h{\restriction}_K$ is continuous. Given any compact subset $K\subseteq X$, consider the $\sigma$-compact subgroup $H$ of $X$ generated by the compact set $K$. By \cite[3.4.6]{AT}, the $\sigma$-compact group $H$ is $\w$-narrow and so is its closure $\bar H$ in $X$, see \cite[3.4.9]{AT}. Since closed subspaces of \v Cech-complete spaces are \v Cech-complete \cite[3.9.6]{Eng}, the topological group $\bar H$ is \v Cech-complete. The universal BP-measurability of the homomorphism $h$ implies the BP-measurability of the restriction $h{\restriction}_{\bar H}:\bar H\to Y$. By Theorem~\ref{t:main-BP}, the homomorphism $h{\restriction}_{\bar H}$ is continuous and so is the restriction $h{\restriction}_K$.
\end{proof}

A function $f:X\to Y$ between topological spaces is called {\em Borel-measurable} if for every open set $U\subseteq Y$ the preimage $f^{-1}[U]$ is a Borel subset of $X$. Since Borel-measurable functions are both universally measurable and universally BP-measurable, Theorems~\ref{t:main-UM} or \ref{t:main-uBP} imply our last principal result.

\begin{corollary}\label{c:main} Every Borel-measurable homomorphism $h:X\to Y$ from a \v Cech-complete topological group $X$ to any topological group $Y$ is continuous.
\end{corollary}

\begin{remark} Corollary~\ref{c:main} generalizes an old result of Christensen \cite{Chris71} who proved that under $2^{\w_1}>2^{\w_0}$ every Borel-measurable homomorphism $h:X\to Y$ from a first-countable \v Cech-complete Abelian topological group $X$ to any topological group $Y$ is continuous.
\end{remark}

We do not know whether the $\w$-narrowness of $X$ can be removed from Theorem~\ref{t:main-BP}.

\begin{problem} Let $h:X\to Y$ be a BP-measurable homomorphism from a \v Cech-complete topological group $X$ to a topological group $Y$. Is $h$ continuous?
\end{problem}

\begin{remark} The \v Cech-completeness cannot be removed from Theorems~\ref{t:main-UM}, \ref{t:main-BP}, \ref{t:main-uBP} or Corollary~\ref{c:main}: the homomorphism $$h:\IQ+\IQ\sqrt{2}\to\IQ+\IQ\sqrt{3},\quad  h:x{+}y\sqrt{2}\,\mapsto\, x{+}y\sqrt{3},$$between the countable dense subgroups $\IQ+\IQ\sqrt{2}$ and $\IQ+\IQ\sqrt{3}$ of the real line is Borel-measurable and discontinuous.
\end{remark}

Theorems~\ref{t:main-H} and \ref{t:main-BP} will be deduced from more powerful Theorems~\ref{t:KA2}, \ref{t:KA3} on the automatic continuity of $\AI$-measurable homomorphisms on Baire $K$-analytic groups, proved in Section~\ref{s:5} after some preliminary work made in Sections~\ref{s:2}--\ref{s:4}. In Section~\ref{s:6} we characterize Baire $K$-analytic groups as \v Cech-complete groups which are Lindel\"of, $\w$-narrow or countably cellular. Sections~\ref{s:7}, \ref{s:8}, \ref{s:9} contain the proofs of Theorems~\ref{t:main-BP}, \ref{t:main-H}, \ref{t:main-UM}, respectively.

\section{$K$-analytic spaces}\label{s:2}

All topological spaces considered in this paper are assumed to be Hausdorff.
%A topological group $X$ is called {\em $K$-analytic} if its underlying topological space is $K$-analytic. The class of $K$-analytic spaces is a nice class of topological spaces that includes all compact Hausdorff spaces and has many nice inheritance properties. 

A Tychonoff space $X$ is called
\begin{itemize}
\item {\em Lindel\"of} if every open cover of $X$ has a countable subcover;
\item {\em \v Cech-complete} if $X$ is homeomorphic to a $G_\delta$-set in some compact Hausdorff space;
\item {\em $K$-analytic} if there exists a continuous surjective map $f:Z\to X$ defined on a Lindel\"of \v Cech-complete space $Z$;
\item {\em analytic} if there exists a continuous surjective map $f:Z\to X$ defined on a Polish space $Z$;
\item {\em cosmic} if there exists a continuous surjective map $f:Z\to X$ defined on a separable metrizable space $Z$.
\end{itemize}

Theorem 2.6.1 in \cite{RJ} implies that in the class of Tychonoff spaces our definition of a $K$-analytic space is equivalent to the original definition (via upper semicontinuous compact-valued maps) given in \cite{RJ}.

A subset $A$ of a topological space $X$ is called {\em $K$-analytic} if $A$ endowed with the subspace topology is a $K$-analytic space.

In the following lemma we collect some properties of $K$-analytic spaces that will be used in the subsequent proofs.

\begin{lemma}\label{l:KA}
\begin{enumerate}\itemsep=2pt
\item A $K$-analalytic space $X$ is analytic if and only if $X$ is cosmic.
\item Any $K$-analytic subspace of any Tychonoff space $X$ has the Baire property in $X$. 
\item A subspace $X$ of a (compact) Hausdorff space $Y$ is $K$-analytic (if and) only if there exists a countable family $(F_s)_{s\in\w^{<\w}}$ of closed subsets of $Y$ such that $X=\bigcup_{s\in\w^\w}\bigcap_{n\in\w}F_{s{\restriction}_n}$.
\item For any continuous map $f:X\to Y$ between $K$-analytic spaces and any $K$-analytic subset $A\subseteq Y$ the preimage $f^{-1}[A]$ is $K$-analytic.
\end{enumerate}
\end{lemma}  

\begin{proof} 1. The first statement is proved in Theorem 5.5.1 of \cite{RJ}.
\smallskip 

2. The second statement follows from Theorem 2.5.2 and Corollary 2.9.4 in \cite{RJ}.
\smallskip

3. The third statement can be easily derived from Theorems 2.5.2 and 2.5.4 in \cite{RJ}.
\smallskip

4. Let $f:X\to Y$ be a continuous map between $K$-analytic spaces and $A\subseteq Y$ is a $K$-analytic space. Find Lindel\"of \v Cech-complete spaces $P,Q$ and continuous surjective maps $\varphi:P\to X$ and $\psi:Q\to A$. Since the space $P$ is \v Cech-complete, there exists a compact Hausdorff space $\bar P$ containing the space $P$ as a dense $G_\delta$-subset. Being Lindel\"of, the $G_\delta$-subspace $P$  of the compact Hausdorff space $\bar P$ is equal to the intersection $\bigcap_{n\in\w}P_n$ of a decreasing sequence $(P_n)_{n\in\w}$ of open $\sigma$-compact subsets of $\bar P$. By analogy, the Lindel\"of \v Cech-complete space $Q$ is equal to the intersection $\bigcap_{n\in\w}Q_n$ of a decreasing sequence $(Q_n)_{n\in\w}$ of open $\sigma$-compact sets in some compact Hausdorff space $\bar Q$. Then $P\times Q$ is equal to the intersection $\bigcap_{n\in\w}(P_n\times Q_n)$ of the decreasing sequence $(P_n\times Q_n)_{n\in\w}$ of open $\sigma$-compact sets in the compact Hausdorff space $\bar P\times \bar Q$. By \cite[2.3.3 and 2.5.4]{RJ}, the \v Cech-complete space space $P\times Q$ is $K$-analytic and hence Lindel\"of. Consider the continuous map $\pi:P\times Q\to X$, $\pi:(p,q)\mapsto\varphi(p)$ and observe that $f^{-1}[A]=\pi[F]$ where 
$$F\defeq\{(p,q)\in P\times Q:\varphi(p)=\psi(q)\}$$is a closed subspace of the Lindel\"of \v Cech-complete space $P\times Q$. Since $F$ is Lindel\"of and \v Cech-complete, the space $f^{-1}[A]=\pi[F]$ is $K$-analytic.
\end{proof}

\section{Measurability and semimeasurability in topological spaces}

A subset $A$ of a topological space $X$ is called
\begin{itemize}
\item {\em functionally closed} if $A=f^{-1}(0)$ for some continuous function $f:X\to\IR$;
\item {\em functionally open} if $X\setminus A$ is functionally closed;
\item {\em functionally Borel} if $A=f^{-1}[B]$ for some continuous function $f:X\to \IR^\w$ and some Borel subset $B$ of $\IR^\w$;
\item {\em functionally analytic} if $A=f^{-1}[B]$ for some continuous function $f:X\to \IR^\w$ and some analytic subset $B$ of $\IR^\w$;
\item {\em functionally coanalytic} if $X\setminus A$ is functionally analytic;
\item  {\em functionally arbitrary} if  $A=f^{-1}[B]$ for some continuous function $f:X\to \IR^\w$ and some subset $B$ of $\IR^\w$;
\end{itemize}

Since Borel subsets of Polish spaces are both analytic and coanalytic \cite[14.11]{Ke}, every functionally Borel set is functionally analytic and functionally coanalytic.
\smallskip

A family of sets $\A$ is called a {\em $\sigma$-algebra} if $\bigcup\A\setminus\bigcup \C\in\A$ for any  countable subfamily $\mathcal C\subseteq\A$. Elements of a $\sigma$-algebra $\A$ are called {\em $\A$-measurable} subsets of the set $X=\bigcup\A$.

For a topological space $X$ by $\Bo$ (resp. $\Ba$) we denote the smallest $\sigma$-algebra  containing all (functionally) open subsets of $X$. Elements of the $\sigma$-algebra $\Bo$ (resp. $\Ba$) are called {\em Borel} (resp. {\em Baire}) subsets of $X$. It is easy to see that the $\sigma$-algebra of Baire sets $\Ba$ coincides with the $\sigma$-algebra of functionally Borel sets in $X$.
\smallskip

A family $\I$ of sets is called an {\em ideal} (resp. a {\em $\sigma$-ideal\/}) if it has the following properties:
\begin{itemize}
\item for any finite (resp. countable) subfamily $\C\subseteq\I$, the union $\bigcup\C$ belongs to $\I$;
\item for any sets $A\subseteq B$, the inclusion $A\in\I$ implies $B\in\I$;
\item $\bigcup\I\notin\I$.
\end{itemize}
We shall say that a $\sigma$-ideal $\I$ is defined on a set $X$ if $X=\bigcup\I$. A subset $P$ of $X=\bigcup\I$ is called {\em $\I$-positive} if $P\notin\I$.

We shall say that a $\sigma$-ideal $\I$ on a topological space $X$ has a ({\em functionally}) {\em Borel base} if every set $I\in\I$ is contained in a (functionally)  Borel set $B\in\I$. By analogy we can define $\sigma$-ideals with {\em functionally analytic, functionally coanalytic} or {\em functionally arbitrary base}.  Since every Borel subset of a Polish space is both analytic and coanalytic, every ideal with a functionally Borel base has functionally coanalytic base.
\smallskip

For a $\sigma$-algebra $\A$ and a $\sigma$-ideal $\I$, let $\A^\pm\!\I$ be the smallest $\sigma$-algebra containing the union $\A\cup\mathcal I$.  If $\bigcup\I\subseteq\bigcup\A$, then $$\A^\pm\!\I=\{(A\setminus I)\cup J:A\in\A\mbox{ and }I,J\in\I\}.$$

A family of sets $\F$ is {\em disjoint} if $A\cap B=\emptyset$ for any distinct sets $A,B\in\F$.
\smallskip

A $\sigma$-ideal $\I$ is defined to be {\em $\F$-ccc} for a family of sets $\F$ if any disjoint subfamily of $\F\setminus\I$ is countable. A $\sigma$-ideal on a topological space will be called {\em ccc} if it is $\Bo$-ccc for the family $\Bo$ of all Borel subsets of $X$. Let us mention that {\em ccc} is the abbreviation of the {\em countable chain condition}.
\smallskip

The following (known) proposition can be easily derived from Lemma~\ref{l:KA}(3) and Szpilrajn-Marczewski Theorem 2.9.2 \cite{RJ} on preservation of measurability by the Souslin operation.

\begin{lemma}\label{l:KA=>m} Let $\Bo$ be the $\sigma$-algebra of all Borel subsets of a  Hausdorff space $X$ and $\I$ be a $\sigma$-ideal with Borel base on $X$. If $\I$ is ccc, then every $K$-analytic subspace of $X$ is  $\BoI$-measurable. 
\end{lemma}

\begin{remark} For a ccc $\sigma$-ideal with (functionally) Borel base on a compact Hausdorff space, the $\sigma$-algebra $\BoI$ can be strictly larger than the $\sigma$-algebra $\BaI$. To construct a suitable example, consider the space $X=[0,\w_1]$ endowed with the order topology. A subset $S$ of $[0,\w_1]$ is called {\em stationary} if $S$ has nonempty intersection with any closed uncountable subset of $X$. Since the intersection of countably many closed uncountable sets in $[0,\w_1]$ is uncountable, the family of nonstationary sets is a $\sigma$-ideal and so is the family $\I$ of all subsets of nonstationary functionally Borel sets in $[0,\w_1]$. Using Fodor's Pressing Down Lemma \cite[8.7]{Jech}, it is possible to prove that every stationary Borel set in $[0,\w_1]$ contains an uncountable closed subset of $X$, which implies that the ideal $\I$ is ccc. Since every real-valued continuous function on $[0,\w_1]$ is constant on some neighborhood of $\w_1$, every functionally Borel subset $B\subseteq [0,\w_1)$ of $[0,\w_1]$ is countable. This implies that the open set $[0,\w_1)$ is not $\BaI$-measurable and hence $\Bo\not\subseteq\BaI$ for the compact Hausdorff space $X=[0,\w_1]$.
\end{remark}

\begin{definition} Let $X$ be a topological space and $\I$ be a $\sigma$-ideal.
 A subset $M\subseteq X$ is called {\em $\A^\curlyvee\!\I$-semimeasurable} if for any $K$-analytic set $A\subseteq X$ with $A\cap M\notin\I$ there exists a $K$-analytic $\I$-positive set $B\subseteq A\cap M$.

 A function $f:X\to Y$ to a topological space $Y$ is called {\em $\AI$-semimeasurable} if for every open set $U\subseteq Y$ the preimage $f^{-1}[U]$ is $\AI$-semimeasurable.
\end{definition}

\begin{proposition}\label{p:m=>sm} Let $X$ be a topological space, $\Ba$ be the $\sigma$-algebra of functionally Borel sets in $X$, and $\I$ be a $\sigma$-ideal on $X$. If the ideal $\I$ has a functionally coanalytic base, then every $\BaI$-measurable set  is $\A^\curlyvee\!\I$-semimeasurable.
\end{proposition}

\begin{proof} Assume that the ideal $\I$ has a functionally coanalytic base. Given a set $M\in\BaI$ and a $K$-analytic subspace $A\subseteq X$ with $A\cap M\notin\I$, we should find an $\I$-positive $K$-analytic set in $A\cap M$. Since $M\in\BaI$, there exists a functionally Borel set $B$ in $X$ such that the symmetric difference $M\Delta B=(M\setminus B)\cup(B\setminus M)$ belongs to the ideal $\I$. Since $\I$ has functionally coanalytic base, the symmetric difference $M\Delta B$ is contained in some functionally coanalytic set $C\in\I$. Since the set $B$ is functionally Borel, there exists a continuous function $f:X\to \IR^\w$ such that $B=f^{-1}[B']$ for some Borel (and hence analytic) set in $\IR^\w$. By Lemma~\ref{l:KA}(4), the subset $A\cap B=(f{\restriction}_A)^{-1}[B']$ is $K$-analytic. Since the set  $X\setminus C$ is functionally analytic in $X$, there exists a continuous map $g:X\to\IR^\w$ such that $X\setminus C=g^{-1}[A']$ for some analytic set $A'$ in $\IR^\w$. By Lemma~\ref{l:KA}(4), the set $A\cap B\setminus C=(g{\restriction}_{A\cap B})^{-1}[A']$ is $K$-analytic. Since $A\cap M\notin\I$ and $(A\cap M)\setminus (A\cap B\setminus C)\subseteq C\in\I$, the $K$-analytic subset $A\cap B\setminus C$ of $M$ is $\I$-positive, witnessing that the set $M$ is $\AI$-semimeasurable.
\end{proof}

\begin{proposition} Let $X$ be a $K$-analytic space, $\Bo$ be the $\sigma$-algebra of Borel sets in $X$, $\Ba$ be the $\sigma$-algebra of functionally Borel sets in $X$, and  $\I$ be a ccc $\sigma$-ideal with a functionally Borel  base on $X$ such that $\Bo\subseteq\BaI$. A subset $S$ of $X$ is $\Bo^\pm\!\I$-measurable if and only if the sets $S$ and $X\setminus S$ are $\A^\curlyvee\!\I$-semimeasurable.
\end{proposition}

\begin{proof} Let $S$ be a subset of $X$. If $S$ is $\BoI$-measurable, then it is $\BaI$-measurable as $\Ba\subseteq\BoI$. By Proposition~\ref{p:m=>sm}, the sets $S$ and $X\setminus S$ are $\AI$-semimeasurable.
\smallskip

Now assume that the sets $S$ and $X\setminus S$ are $\AI$-semimeasurable. Applying the Kuratowski--Zorn Lemma, choose maximal disjoint families $\mathcal B'\subseteq\Ba\setminus\I$ and $\mathcal B''\subseteq\Ba\setminus \I$ such that $\bigcup\mathcal B'\subseteq S$ and $\bigcup\mathcal B''\subseteq X\setminus S$. The ccc property of the ideal $\I$ ensures that the families $\mathcal B'$ and $\mathcal B''$ are countable and hence the sets $B'\defeq\bigcup\mathcal B'$ and $B''\defeq X\setminus\bigcup\mathcal B''$ are functionally Borel in $X$. Observe that $B'\subseteq S\subseteq B''$. To see that $S$ is $\BaI$-measurable, it remains to show that $B''\setminus B'\in\I$. By Lemma~\ref{l:KA}(4), the functionally Borel set $B\defeq B''\setminus B'$ in the $K$-analytic space $X$ is $K$-analytic. Assuming that $B\notin\I$, we conclude that $B\cap S\notin\I$ or $B\setminus S\notin\I$. If $B\cap S\notin\I$, then by the $\AI$-semimeasurability of $S$, there exists an $\I$-positive $K$-analytic set $A\subseteq B\cap S$. By Lemma~\ref{l:KA=>m}, the $K$-analytic set $A$ is $\BoI$-measurable. Since $\Bo\subseteq\BaI$, the $\BoI$-measurable set $A$ is $\BaI$-measurable and hence $A\Delta F\in\I$ for some functionally Borel set $F$ in $X$. Since the ideal $\I$ has a functionally Borel base, the set $A\Delta F$ is contained in some  functionally Borel set $I\in\I$. Then the set $A\setminus I=F\setminus I$ is functionally Borel, $\I$-positive, and disjoint with the set $B'\subseteq S$. But the existence of such set contradicts the maximality of the family $\mathcal B'$. This contradiction shows that $B\cap S\in\I$. By analogy, we can use the $\AI$-semimeasurability of the set $X\setminus S$ to prove that $B\setminus S\in\I$. Then $B=(B\cap S)\cup(B\setminus S)\in\I$.
\end{proof} 

The following important theorem was proved by  Brzuchowski, Cicho\'n, Grzegorek and Ryll-Nardzewski  in \cite{BCGRN}.

\begin{theorem}[Brzuchowski, Cicho\'n, Grzegorek, Ryll-Nardzewski]\label{t:4Poles} Let $\I$ be a $\sigma$-ideal with a Borel base on a Polish space $X$. Any point-finite family $\J\subseteq\I$ with $\bigcup\J\notin\I$ contains a subfamily $\J'$ whose union $\bigcup\J'$ is not $\Bo^\pm\!\I$-measurable in $X$.
\end{theorem}

In the proof of our principal results we shall use a ``semi-improvement'' of Theorem~\ref{t:4Poles}, proved by Banakh, Ra\l owski and \.Zeberski in \cite{BRZ}.

\begin{theorem}[Banakh, Ra\l owski, \.Zeberski]\label{t:BRZ} Let $\I$ be a $\sigma$-ideal on an analytic space $X$. Any point-finite family $\J\subseteq\I$ with $\bigcup\J\notin\I$ contains a subfamily $\J'\subseteq\J$ whose union $\bigcup\J'$ is not $\A^\curlyvee\!\I$-semimeasurable.
\end{theorem}

\section{Steinhaus ideals in topological groups}\label{s:4}

For sets $A,B$ is a group $X$, let $$AB\defeq\{ab:a\in A,\;b\in B\}$$ be the pointwise product of the sets $A,B$ in $X$. Also define the powers $A^{\cdot n}$ of $A$ in $X$ by the recursive formula: $$\mbox{$A^{\cdot 1}=A$ and $A^{\cdot(n+1)}=A^{\cdot n}A$ for $n\in\IN$.}$$ 

\begin{definition}  An ideal $\I$ on a topological group $X$ is defined to be
\begin{itemize}
\item {\em left-invariant} if for any set $I\in\I$ and element $x\in X$ the left shift $xI\defeq\{xy:y\in I\}$ of $I$ in the group $X$ belongs to the ideal $\I$;
\item {\em $n$-Steinhaus} for $n\in\IN$ if for any $\I$-positive $K$-analytic set $A$ in $X$, the set $(AA^{-1})^{\cdot n}$ is a neighborhood of the identity in $X$;
\item {\em  Steinhaus} if $\I$ is $n$-Steinhaus for some $n\in\IN$.
\end{itemize}
\end{definition}

\begin{remark} The measurability of $K$-analytic sets and the classical results of Pettis \cite{Pettis} and Steinhaus--Weil \cite{Strom} imply that the ideal $\M$ of meager sets in any Baire topological group is 1-Steinhaus and the ideal $\N$ of Haar-null sets in any locally compact group is 1-Steinhaus. More examples of 1-Steinhaus ideals on Polish groups can be found in \cite{BJGS}, \cite{BO1}, \cite{BO2}, \cite{BO3}, \cite{Chris72}, \cite{EN}, \cite{Jab}.
\end{remark}

A topological space $X$ is {\em Baire} if for any sequence $(U_n)_{n\in\w}$ of open dense sets in $X$, the intersection $\bigcap_{n\in\w}U_n$ is dense in $X$. A topological group is {\em Baire} if its underlying topological space is Baire.

In the proof of Proposition~\ref{p:Steinhaus} we shall use the following classical result of Pettis \cite{Pettis}.

\begin{lemma}[Pettis]\label{l:Pettis} Let $A$ be a nonmeager set in a Baire topological group $X$. If $A$ has the Baire property in $X$, then $AA^{-1}$ is a neighborhood of the identity in $X$.
\end{lemma} 

\begin{proposition}\label{p:Steinhaus} Let $X$ be a Baire topological group and $\A$ be the family of all $K$-analytic sets in $X$. Any left-invariant $\A$-ccc $\sigma$-ideal $\I$ on $X$ is $2$-Steinhaus.
\end{proposition}

\begin{proof} Let $A$ be an $\I$-positive $K$-analytic set in $X$. Since $\I$ is left-invariant, the family $\{xA\}_{x\in X}$ consists of $\I$-positive $K$-analytic sets in $X$. Using Kuratowski--Zorn Lemma, choose a maximal subset $M\subseteq X$ such that $xA\cap yA=\emptyset$ for any distinct elements $x,y\in M$. The $\A$-ccc property of $\I$ ensures that the set $M$ is countable. The maximality of $M$ implies that for every $x\in X$ the set $xA$ intersects the set $MA$  and hence $X=MAA^{-1}$. Since the space $X$ is Baire, the set $AA^{-1}$ is not meager. By Theorem 2.5.5 in \cite{RJ}, the space $A\times A$ is $K$-analytic and hence the set $AA^{-1}$ is $K$-analytic, being the image of the $K$-analytic space $A\times A$ under the continuous map $A\times A\to AA^{-1}$, $(x,y)\mapsto xy^{-1}$. By Lemma~\ref{l:KA}(2), the $K$-analytic set $AA^{-1}$ has the Baire property in $X$. Since $AA^{-1}$ is not meager, we can apply Lemma~\ref{l:Pettis} and conclude that the set $(AA^{-1})^{\cdot 2}=AA^{-1}(AA^{-1})^{-1}$ is a neighborhood of the identity in $X$.
\end{proof}

\begin{lemma}\label{l:Accc} Every ccc $\sigma$-ideal $\I$ with a Borel base on a  Hausdorff space $X$ is $\A$-ccc for the family $\A$ of all $K$-analytic subsets of $X$.
\end{lemma}

\begin{proof} Given any disjoint family $\mathcal D$ of $\I$-positive $K$-analytic sets in $X$, we should prove that $\mathcal D$ is countable. By Lemma~\ref{l:KA=>m}, every $K$-analytic set $D\in\mathcal D$ is $\BoI$-measurable and hence $D\Delta B_D\in\I$ for some Borel set $B_D$ in $X$. Since the ideal $\I$ has Borel base, the set $D\Delta B_D$ is contained in a Borel set $I_D\in\I$. Then $B_D\setminus I_D$ is an $\I$-positive Borel set in $D$ and $\{B_D\setminus I_D\}_{B\in\mathcal D}$ is a disjoint family of $\I$-positive Borel sets in $X$. Since the ideal $\I$ is ccc, this family is countable and so is the family $\mathcal D$.
\end{proof} 

Proposition~\ref{p:Steinhaus} and Lemma~\ref{l:Accc} imply the following corollary.

\begin{corollary}\label{c:Steinhaus} Every left-invariant ccc $\sigma$-ideal $\I$ with a Borel base on a Baire topological group $X$ is $2$-Steinhaus.
\end{corollary}

\section{$\AI$-semimeasurable homomorphisms on $K$-analytic groups}\label{s:5}

A topological group is called {\em analytic} (resp. {\em $K$-analytic}) if so is its underlying topological space.

\begin{proposition}\label{p:a} Let $\I$ be a left-invariant $\sigma$-ideal on an analytic group $X$ and $h:X\to Y$ be an $\A^\curlyvee\!\I$-semimeasurable homomorphism onto a topological group $Y$. Then for every neighborhood $U\subseteq Y$ of the identity in $Y$ the preimage $h^{-1}[U]$ is $\I$-positive.
\end{proposition}

\begin{proof} To derive a contradiction, assume that $h^{-1}[U]\in \I$ for some open neighborhood  $U\subset Y$ of the identity $e$ of the group $Y$.  
By Markov's Theorem \cite[3.9]{AT}, there exists a left-invariant continuous pseudometric $\rho $ on $Y$ such that $\{y\in Y:\rho(y,e)<1\}\subseteq U$. The pseudometric $\rho$ determines an equivalence relation $\sim$ on $Y$ such that $x\sim y$ iff $\rho(x,y)=0$. Let $q:Y\to \tilde Y$ be the quotient map to the quotient set $\tilde Y=Y/_\sim$. The pseudometric $\rho$ determines a unique  metric $\tilde\rho$ on $\tilde Y$ such that $\tilde\rho(q(x),q(y))=\rho(x,y)$ for all $x,y\in Y$.  By the paracompactness of the metric space $(\tilde Y,\tilde\rho)$, there exists a $\sigma$-discrete cover $\tilde\V$ of $\tilde Y$ by open sets of $\tilde\rho$-diameter $<1$.
Then $\V=\{q^{-1}[V]:V\in\tilde\V\}$ is a $\sigma$-discrete cover of $Y$ by open sets of $\rho$-diameter $<1$. 

 For every $V\in\V$ chose a point $v\in V$ and observe that $v^{-1}V\subseteq \{y\in Y:\rho(y,e)<1\}\subseteq U$. Choose any point $x\in h^{-1}(v)$ (which exists by the surjectivity of $h$) and conclude that $x^{-1}h^{-1}[V]=h^{-1}[v^{-1}V]\subseteq h^{-1}[U]\in\I$ and $h^{-1}[V]\in\I$ (by the left-invariance of the ideal $\I$).
 
 Write the $\sigma$-discrete family $\V$ as the countable union $\bigcup_{n\in\w}\V_n$ of discrete families $\V_n$. Since $X=\bigcup_{n\in\w}h^{-1}[\textstyle\bigcup\V_n]\notin\I$, for some $n\in\w$ the set $h^{-1}[\bigcup\V_n]=\bigcup_{V\in\V_n}h^{-1}[V]$ does not belong to the $\sigma$-ideal $\I$. By Theorem~\ref{t:BRZ}, for some subfamily $\V'\subseteq\V_n$ the union $\bigcup_{V\in\V'}h^{-1}[V]=h^{-1}[\bigcup\V']$ is not $\AI$-semimeasurable, which contradicts the $\AI$-semimeasurability of the homomorphism $h$.
\end{proof}

Now we extend Proposition~\ref{p:a} to $K$-analytic groups. 

\begin{proposition}\label{p:KA} Let  $\I$ be a left-invariant $\sigma$-ideal with a functionally arbitrary base on a $K$-analytic group $X$.  For any $\AI$-semimeasurable homomorphism $h:X\to Y$ to a topological group $Y$ and every neighborhood $U\subseteq Y$ of the identity, the preimage $h^{-1}[U]$ is $\I$-positive.
\end{proposition}

\begin{proof} To derive a contradiction, assume that $h^{-1}[U]\in\I$ for some open neighborhood $U$ of the identity $e_Y$ in the topological group $Y$. Since the ideal $\I$ has a functionally arbitrary base, for the set $P\defeq h^{-1}[U]\in\I$ there exists a continuous function $f:X\to \IR^\w$ such that $f^{-1}\big[f[P]\big]\in \I$. By the Lindel\"of property of the $K$-analytic group $X$ and Tkachenko's Theorem~\cite[8.1.6]{AT} on $\IR$-factirizability of Lindel\"of topological groups, there exists a continuous homomorphism $p:X\to G$ onto a metrizable separable topological group $G$ and a continuous function $g:G\to \IR^\w$ such that $f=g\circ p$. Then $p^{-1}[p[P]]\subseteq f^{-1}[f[P]]\in\I$. Being a continuous image of the $K$-analytic group $X$, the topological group $G$ is $K$-analytic. By Lemma~\ref{l:KA}(1), the metrizable separable topological space $G$ is analytic. Let $K=p^{-1}(e_G)$ be the kernel of the homomorphism $p$.

Replacing $Y$ by its subgroup $h[X]$, we can assume that the homomorphism $h:X\to Y$ is surjective. Then the subgroup $h[K]$ of $Y=h[X]$ is normal and so is its closure $H=\overline{h[K]}$ in $Y$. Choose an open neighborhood $V$ of the identity in $Y$ such that $V^{-1}V\subseteq U$ and observe that
$$HV=\overline{h[K]}V\subseteq h[K]V^{-1}V\subseteq h[K]U.$$

Let $Z=Y/H$ be the quotient group of the topological group $Y$ and $q:Y\to Z$ be the quotient homomorphism. Since the kernel of the homomorphism $p:X\to G$ is contained in the kernel of the homomorphism $q\circ h:X\to Z$, there exists a unique homomorphism $\hbar:G\to Z$ such that $\hbar \circ p=q\circ h$. We claim that the homomorphism $\hbar $ is $\A^\curlyvee\!\J$-semimeasurable for the left-invariant $\sigma$-ideal 
$$\J\defeq\{J\subseteq G:p^{-1}[J]\in\I\}$$
on the  group $G$. Given any open set $W\subseteq Z$ and a $K$-analytic subspace $A\subseteq G$ with $A\cap \hbar^{-1}[W]\notin\J$, we should find a $K$-analytic subspace $A'\subseteq A\cap \hbar^{-1}[W]$ such that $A'\notin\J$. By Lemma~\ref{l:KA}(4), the preimage $p^{-1}[A]$ is a $K$-analytic subspace of the $K$-analytic group $X$. It follows from $A\cap \hbar^{-1}[W]\notin\J$ and $q\circ h=\hbar\circ p$ that $$p^{-1}[A]\cap h^{-1}[q^{-1}[W]]=p^{-1}[A\cap \hbar^{-1}[W]]\notin\I.$$ Since the homomorphism $h$ is $\AI$-semimeasurable, there exists a $K$-analytic subspace $A''\subseteq p^{-1}[A]\cap h^{-1}[q^{-1}[W]]$ such that $A''\notin \I$. Then $A'=p[A'']$ is a $K$-analytic subspace of $A\cap \hbar^{-1}[W]$ such that $A'\notin\J$. This completes the proof of the $\A^\curlyvee\!\J$-semimeasurability of the homomorphism $\hbar$. 

By Proposition~\ref{p:a}, 
$\hbar^{-1}[q[V]]\notin\J$ and hence $p^{-1}\big[\hbar^{-1}[q[V]]\big]\notin\I$.
On the other hand,
$$p^{-1}[\hbar^{-1}[q[V]]=(q\circ h)^{-1}[q[V]]=h^{-1}[HV]\subseteq h^{-1}[h[K]U]=Kh^{-1}[U]=KP=p^{-1}[p[P]]\in\I,$$
which is a desired contradiction that completes the proof.
\end{proof}

\begin{proposition}\label{p:KAS} Let  $\I$ be a Steinhaus left-invariant $\sigma$-ideal with a functionally arbitrary base on a $K$-analytic group $X$. Every $\AI$-semimeasurable homomorphism $h:X\to Y$ to a topological group $Y$ is continuous.
\end{proposition}

\begin{proof} The continuity of the homomorphism $h$ will follow as soon as we check that for any open neighborhood $U$ of the identity in $Y$ the preimage $h^{-1}[U]$ is a neighborhood of the identity in $X$. By our assumption, the ideal $\I$ is Steinhaus and hence $n$-Steinhaus for some $n\in\IN$. By the continuity of the multiplication and inversion in $Y$, there exists an open neighborhood $V$ of the identity  in $Y$ such that $(VV^{-1})^{\cdot n}\subseteq U$. Proposition~\ref{p:KA} ensures that $h^{-1}[V]\notin\I$. By the $\AI$-semimeasurability of $h$, the set $h^{-1}[V]\notin\I$ contains an $\I$-positive $K$-analytic set $A$. The $n$-Steinhaus property of the ideal $\I$ ensures that the set $(AA^{-1})^{\cdot n}$ is a neighborhood of the identity in $X$. It follows from $h[A]\subseteq V$ that $h[(AA^{-1})^{\cdot n}]\subseteq (VV^{-1})^{\cdot n}\subseteq U$ and hence $h^{-1}[U]\supseteq (AA^{-1})^{\cdot n}$ is a neighborhood of the identity in $X$.
\end{proof}

\begin{theorem}\label{t:KA2} Let  $\I$ be a  left-invariant Steinhaus $\sigma$-ideal with a functionally coanalytic base on a $K$-analytic group $X$ such that $\Bo\subseteq\BaI$. For a homomorphism $h:X\to Y$ to a topological group $Y$ the following conditions are equivalent:
\begin{enumerate}
\item $h$ is continuous;
\item $h$ is $\BoI$-measurable;
\item $h$ is $\BaI$-measurable;
\item $h$ is $\AI$-semimeasurable.
\end{enumerate}
\end{theorem}

\begin{proof} The implication $(1)\Ra(2)$ is trivial and $(2)\Ra(3)$ follows from the assumption $\Bo\subseteq\BaI$, which implies $\BoI=\BaI$. The implications $(3)\Ra(4)\Ra(1)$ follow from Propositions~\ref{p:m=>sm} and Proposition~\ref{p:KAS}.
\end{proof}

Corollary~\ref{c:Steinhaus} and Theorem~\ref{t:KA2} imply the following theorem that will be used in the proofs of Theorems~\ref{t:main-H} and \ref{t:main-BP}.

\begin{theorem}\label{t:KA3} Let  $\I$ be a  left-invariant ccc $\sigma$-ideal with a functionally Borel base on a Baire $K$-analytic group $X$ such that $\Bo\subseteq\BaI$. For a homomorphism $h:X\to Y$ to a topological group $Y$ the following conditions are equivalent:
\begin{enumerate}
\item $h$ is continuous;
\item $h$ is $\BoI$-measurable;
\item $h$ is $\BaI$-measurable;
\item $h$ is $\AI$-semimeasurable.
\end{enumerate}
\end{theorem}

\section{Characterizing Baire $K$-analytic groups}\label{s:6}

Theorem~\ref{t:KA3} motivates the problem of deeper studying the structure of Baire $K$-analytic groups. We shall prove that the class of such groups coincides with the class of \v Cech-complete groups which are $\w$-narrow, Lindel\"of or countably cellular. %By a {\em \v Cech-complete group} we understand a topological group whose underlying topological space is \v Cech-complete. 

A topological space $X$
\begin{itemize}
\item is {\em countably cellular} if every disjoint family of open sets in $X$ is countable;
\item has {\em countable pseudocharacter} if  each singleton $\{x\}\subseteq X$ is a $G_\delta$-set in $X$.
\end{itemize}

\begin{theorem}\label{t:Cech} For a topological group $X$ the following conditions are equivalent:
\begin{enumerate}
\item $X$ is $K$-analytic and Baire;
\item $X$ contains a compact normal subgroup $H$ such that the quotient group $X/H$ is Polish;
\item $X$ is $\w$-narrow and \v Cech-complete;
\item $X$ is Lindel\"of and \v Cech-complete;
\item $X$ is countably cellular and \v Cech-complete.
\end{enumerate}
If the group $G$ has countable pseudocharacter, then the conditions \textup{(1)--(5)} are equivalent to
\begin{itemize}
\item[(6)] $X$ is Polish.
\end{itemize}
\end{theorem}

\begin{proof} $(1)\Ra(2)$ Assume that $X$ is $K$-analytic and Baire. Let $\beta X$ be the Stone-\v Cech compactification of $X$. By Lemma~\ref{l:KA}(2), the $K$-analytic set $X$ has the Baire property in $\beta X$. Then there exists an open set $U$ in $\beta X$ such that the symmetric difference $U\triangle X$ is contained in a meager $F_\sigma$-set $M$ in $\beta X$. Since $X$ is dense in $\beta X$, the intersection $X\cap M$ is meager in $X$. Since the space $X$ is Baire, the complement $X\setminus M$ is dense in $X$ and hence dense in $\beta X$. Since $X\setminus M\subseteq U$, the open set $U$ is dense in $\beta X$ and hence $\beta X\setminus U$ is nowhere dense in $\beta X$. Replacing $M$ by $M\cup(\beta X\setminus U)$, we can assume that $(\beta X\setminus U)\subseteq M$ and hence $\beta X\setminus M\subseteq X$ is a dense $G_\delta$-subset of $\beta X$, contained in $X$. Therefore, the topological group $X$ contains a dense \v Cech-complete subspace $G\defeq \beta X\setminus M$. Replacing $G$ by a suitable shift of $G$, we can assume that $G$ contains the identity $e$ of the group $X$. Write $G$ as $G=\bigcap_{n\in\w}W_n$ for a decreasing sequence $(W_n)_{n\in\w}$ of open subsets of $\beta X$. By the complete regularity of the compact Hausdorff space $\beta X$, for every $n\in\w$ there exists a neighborhood $K_n$ of $e$ in $\beta X$ such that $K_n\subseteq W_n$ and $K_n$ is a compact $G_\delta$-set in $\beta X$.

The topological group $X$, being $K$-analytic, is Lindel\"of and hence $\w$-narrow. By \cite[3.4.19]{AT}, for every $n\in\w$ exists a closed normal $G_\delta$-subgroup $H_n$ in $X$ such that $H_n\subseteq K_n$. Let $\bar H_n$ be the closure of $H_n$ in $\beta X$. Then $$\bigcap_{n\in\w}\bar H_n\subseteq \bigcap_{n\in\w}K_n\subseteq \bigcap_{n\in\w}W_n=G\subseteq X$$ is a compact subset of $X$. Since 
$$\bigcap_{n\in\w}\bar H_n=X\cap\bigcap_{n\in\w}\bar H_n=\bigcap_{n\in\w}(X\cap \bar H_n)=\bigcap_{n\in\w}H_n,$$
the intersection $H=\bigcap_{n\in\w} H_n=\bigcap_{n\in\w}\bar H_n$ is a compact normal $G_\delta$-subgroup of $X$. Since $H$ is a closed $G_\delta$-set in $\bigcap_{n\in\w}K_n\subseteq G\subseteq X$ and $\bigcap_{n\in\w} K_n$ is a compact $G_\delta$-set in $\beta X$, the compact set $H$ is of type $G_\delta$ in $\beta X$. Then $H=\bigcap_{n\in\w}V_n$ for some sequence $(V_n)_{n\in\w}$ of open sets in $\beta X$ such that $\overline V_{n+1}\subseteq V_n$ for all $n\in\w$. By \cite[4.3.2]{AT}, every open neighborhood of $H$ in $\beta X$ contains some set $V_n$. Consequently, every open neighborhood of $H$ in $X$ contains some set $X\cap V_n$. This implies that the quotient group $Y=X/H$ is first-countable and hence metrizable by the Birkhoff-Kakutani Theorem \cite[3.3.12]{AT}. By Lemma~\ref{l:KA}(1), the quotient group $Y=X/H$ is analytic. Taking into account that the quotient homomorphism $q:X\to X/H$ is open and the space $X$ is Baire, we conclude that $Y$ is Baire, too. Let $\bar Y$ be the Ra\u\i kov completion of the topological group $Y$. Since $Y$ is metrizable and separable, the topological group $\bar Y$ is Polish. By Lemma~\ref{l:KA}(2), the analytic subgroup $Y$ of $\bar Y$ has the Baire property in $Y$. Being Baire, the $BP$-set $Y$ contains a dense $G_\delta$-subset $D$ of $\bar Y$. Assuming that $Y\ne \bar Y$, we can choose a point $y\in \bar Y\setminus Y$ and conclude that $D$ and $yD$ are two disjoint dense $G_\delta$-sets in the Polish space $\bar Y$, which contradicts the Baire Theorem. This contradiction shows that the topological group $Y=\bar Y$ is Polish.
\smallskip

$(2)\Ra(3)$ Assume that $X$ contains a compact normal subgroup $H$ such that the quotient group $X/H$ is Polish. By \cite[4.3.18]{AT}, the topological group $X$ is \v Cech-complete. By Theorem~\cite[1.5.7]{AT}, the compactness of the subgroup $H$ of $X$ implies that the quotient map $q:X\to X/H$ is closed. By Theorem 3.8.8 of \cite{Eng}, the Lindel\"of property of the Polish space $X/H$ implies that the space $X$ is Lindel\"of. By \cite[3.4.6]{AT}, the Lindel\"of topological group $X$ is $\w$-narrow.
\smallskip

$(3)\Ra(4)$ Assume that $X$ is an $\w$-narrow \v Cech-complete group. By Corollary 4.3.5 in \cite{AT}, the neutral element $e$ of $X$ has a countable family $(W_n)_{n\in\w}$ of open neighborhoods such that the set $K=\bigcap_{n\in\w}W_n$ is compact, $\overline{W}_{n+1}\subseteq W_n$ for all $n\in\w$, and every neighborhood of $K$ in $X$ contains some set $W_n$. By \cite[3.4.19]{AT}, $X$ contains a closed  normal $G_\delta$-subgroup $H$ such that $H\subseteq K$ and hence the subgroup $H$ is compact. Since $H$ is a $G_\delta$-set in $X$, there exists a decreasing sequence $(U_n)_{n\in\w}$ of open neighborhoods of $H$ in $X$ such that $H=\bigcap_{n\in\w}U_n$ and $\overline U_{n+1}\subseteq U_n\subseteq W_n$ for all $n\in\w$. We claim that every neighborhood $U$ of $H$ in $X$ contains some set $U_n$. By the compactness of $K$ and the equality $H=\bigcap_{n\in\w}(K\cap\overline U_n)$, there exists $k\in\w$ such that $K\cap\overline U_k\subseteq U$. By the choice of the sequence $(W_n)_{n\in\w}$, the open neighborhood $U\cup(X\setminus\overline U_k)$ of $K$ contains some set $W_n$ with $n\ge k$. Then $U_n\subseteq W_n\subseteq U\cup (X\setminus\overline U_k)$ and hence $U_n\subseteq U$. The sequence $(U_n)_{n\in\w}$ and the openness of the quotient homomorphism $q:X\to X/H$ witness that the topological group $X/H$ is first-countable. It is also $\w$-narrow, being a  homomorphic image of the $\w$-narrow group $X$. By \cite[3.4.5]{AT}, the first-countable $\w$-narrow topological group $X$ is second-countable and hence Lindel\"of. By Theorem~\cite[1.5.7]{AT}, the compactness of the subgroup $H$ of $X$ implies that the quotient map $q:X\to X/H$ is closed. By Theorem 3.8.8 of \cite{Eng}, the Lindel\"of property of the Polish space $X/H$ implies that the space $X$ is Lindel\"of.
\smallskip

$(4)\Ra(1)$ If the topological group $X$ is Lindel\"of and \v Cech-complete, then it is $K$-analytic and Baire, see \cite[3.9.4]{Eng}.
\smallskip

$(5)\Ra(3)$ If the topological group $X$ is cellular (and \v Cech-complete), then it is $\w$-narrow (and \v Cech-complete) by \cite[3.4.7]{AT}.
\smallskip

$(2)\Ra(5)$ Assume that $X$ contains a compact subgroup $H$ such that the quotient space $X/H$ is Polish. To prove that $X$ has countable cellularity, fix any family $\mathcal U$ consisting of pairwise disjoint nonempty open sets in $X$. Let $D$ be a countable dense set in the Polish space $X/H$. Since the quotient map $q:X\to X/H$ is open, for every $U\in \U$ there exists $y\in D$ such that $y\in q[U]$. Then $\U=\bigcup_{y\in D}\U_y$ where $\U_y=\{U\in\U:U\cap q^{-1}(y)\ne\emptyset\}$ for $y\in D$. For every $y\in D$ the compact space $q^{-1}(y)$ is homeomorphic to the compact topological group $H$, which has countable cellularity by \cite[4.1.8]{AT}.  Then the family $\U_y$ is countable and so is the union $\U=\bigcup_{y\in D}\U_y$.
\smallskip

Now assuming that the space $X$ has countable pseudocharacter, we shall prove that $(2)\Leftrightarrow(6)$. In fact, the implication $(6)\Ra(2)$ is trivial. To prove that $(2)\Ra(6)$, assume that $X$ contains a compact normal subgroup $H$ such that the quotient group $X/H$ is Polish. Since $X$ has countable pseudocharacter, the compact subgroup $H$ has countable (pseudo)character and hence is metrizable, see \cite[3.3.17]{AT}.  By Vilenkin Theorem \cite[3.3.20]{AT}, the topological group $X$ is metrizable. Since $(2)\Leftrightarrow(4)$, the topological group $X$ is Lindel\"of and \v Cech-complete. Being metrizable, the Lindel\"of \v Cech-complete space $X$ is Polish by \cite[4.3.26]{Eng}.
\end{proof}

\begin{remark} Theorem~\ref{t:Cech} generalizes results of Banakh, Ravsky \cite{BR} (and Christensen \cite[5.4]{Chris}) who proved that any Baire analytic group is Polish (if the topology of $X$ is generated by an invariant metric).
\end{remark}

\section{Proof of Theorem~\ref{t:main-BP}}\label{s:7}

Theorem~\ref{t:main-BP} can be easily derived from Theorems~\ref{t:KA3}, \ref{t:Cech} and 
the following lemma.

\begin{lemma}\label{l:M-fBb} Let $\M$ be the $\sigma$-ideal of meager sets in a countably cellular Tychonoff space $X$. Then
\begin{enumerate}
\item $\M$ has a functionally Borel base;
\item $\M$ is ccc;
\item $\Bo^\pm\!\M=\Ba^\pm\!\M$ is the $\sigma$-algebra of sets with the Baire Property in $X$.
\end{enumerate}
\end{lemma}

\begin{proof} 1. To show that the ideal $\M$ has a functional base, take any meager set $M$ in $X$ and find a sequence $(M_n)_{n\in\w}$ of closed nowhere dense sets in $X$ such that $X\subseteq \bigcup_{n\in\w}M_n$. Let $\mathcal F$ be the family of all nonempty functionally open sets in $X$. For every $n\in\w$, let $\F_n=\{U\in\F:U\cap M_n=\emptyset\}$. Using Kuratowski--Zorn Lemma, choose a maximal subfamily $\U_n\subseteq\F_n$ that consists of pairwise disjoint sets. Since $X$ has countable cellularity, the family $\U_n$ is countable. Then its union $U_n=\bigcup\U_n$ is functionally open set in $X$. The maximality of $\U_n$ ensures that the set $U_n$ is dense in $X$ and hence $X\setminus U_n$ is a functionally closed nowhere dense subset of $X$. Then the union $\bigcup_{n\in\w}(X\setminus U_n)$ is a functionally Borel meager subset of $X$ that contains the meager set $M$ and witnesses that the ideal $\I$ has a functionally Borel base.
\smallskip

2. To show that the ideal $\M$ is ccc, take any disjoint family $\mathcal D$ of nonmeager Borel sets in $X$. Since Borel sets have the Baire Property, for every $D\in \mathcal D$ there exists an open set $U_D$ in $X$ such that the symmetric difference $U_D\Delta D$ is meager in $X$. Since the set $D$ is nonmeager, the open set $U_D$ is nonmeager, too. Let $V_D$ be the union of all open Baire subpaces in $U_D$. Since $V_D$ is the largest Baire open subspace of $U_D$, the complement $U_D\setminus V_D$ is meager and hence the set $V_D$ is not meager. It follows that the set $M_D=V_D\setminus D$ is meager. 

We claim that the family $(V_D)_{D\in\mathcal D}$ consists of pairwise disjoint open sets. Indeed, assuming that $V_D\cap V_{D'}\ne\emptyset$ for some distinct sets $D,D'\in\mathcal D$, we conclude that the $V_D\cap V_{D'}\cap(M_D\cup M_{D'})$ is meager in the nonempty Baire space $V_D\cap V_{D'}$ and hence the set $V_D\cap V_{D'}\setminus (M_D\cup M_{D'})$ is not empty and thus contains some point $x$. The point $x$ belongs to $V_D\setminus (M_D\cup M_{D'})\subseteq V_D\setminus M_D=V_D\cap D\subseteq D$ also to $D'$, which is not possible as the sets $D,D'$ are disjoint. This contradiction shows that the family $\{V_D\}_{D\in\mathcal D}$ consists of pairwise disjoint open sets. Since the topological space $X$ has countable cellularity, this family is countable and so is the family $\mathcal D$.
\smallskip

3. By definition, the $\sigma$-algebra $\BoI$ coincides with the $\sigma$-algebra of sets with the Baire Property in $X$. To see that $\BoI=\BaI$, it suffices to check that every open set $U\subseteq X$ belongs to the $\sigma$-algebra $\BaI$. Using the Kuratowski--Zorn Lemma, choose a maximal disjoint family $\U$ of functionally open sets in $X$ such that $\bigcup\U\subseteq U$. The countable cellularity of $X$ ensures that the family $\U$ is countable and the maximality of $\U$ guarantees that the union $\bigcup\U$ is dense in $U$. Then $\bigcup\U$ is a functionally open set in $X$ such that $U\setminus \bigcup\U$ is nowhere dense in $X$, witnessing that $U\in\BaI$.
\end{proof}

\begin{remark} The ideal of meager sets in the compact Hausdorff space $\beta\w\setminus\w$ fails to have a functionally arbitrary base (since each nonempty $G_\delta$-set in $\beta\w\setminus\w$ has nonempty interior). This example shows that the countably cellularity of $X$ cannot be removed from the formulation of  Lemma~\ref{l:M-fBb}. 
\end{remark}

\section{Proof of Theorem~\ref{t:main-H}}\label{s:8}

In this section we apply Theorem~\ref{t:KA3} to prove Theorem~\ref{t:main-H} on the automatic continuity of Haar-measurable homomorphisms on locally compact groups. 

Let us recall that a {\em Haar measure} of a topological group $X$ is any nontrivial left-invariant $\sigma$-additive Borel measure $\lambda:\Bo(X)\to [0,\infty]$  such that 
\begin{itemize}
\item $\lambda(K)<\infty$ for every compact set $K\subseteq X$;
\item for any Borel set $B\subseteq X$ and any real number $a<\lambda(B)$ there exists a compact set $K\subseteq B$ such that $\lambda(K)\ge a$.
\end{itemize}
The last condition is called the {\em inner regularity} of the Haar measure.
It is well-known \cite{Alfsen}, \cite[Ch.44]{Fremlin} that a topological group has a Haar measure if and only if it is locally compact. Moreover, any two Haar measures on a locally compact group differ by a positive multiplier. In this sense a Haar measure on a locally compact group is unique.
\smallskip

A subset $A$ of a locally compact group $X$ is called  {\em Haar-null} if $A\subseteq B$ for some Borel set $B$ of Haar measure zero. 
\smallskip

Now our strategy is to prove that for a $\sigma$-compact locally compact topological group $X$ the ideal $\N$ of Haar-null sets satisfies the requirements of Theorem~\ref{t:KA3}.

%
%\item {\em Haar-measurable} if there exists a Borel set $B\subseteq X$ such that the symmetric difference $A\triangle B$ is Haar-null in $X$. 
%\end{itemize}
%A function $f:X\to Y$ from a locally compact group $X$ to a topological space $Y$ is called {\em Haar-measurable} if for any open set $U\subseteq Y$ the preimage $f^{-1}[U]$ is a Haar-measurable set in $X$. It is clear that any continuous function on a locally compact group is Haar-measurable. The converse is true for Haar-measurable homomorphisms on locally compact groups, see Theorem~\ref{t:HM} below. To prove this theorem we need the following lemmas.

\begin{lemma}\label{l:Hn-fB}  Each compact set $K$ in a locally compact group $X$  contains a functionally closed subset $F$ of $X$ such that $K\setminus F$ is Haar-null.
\end{lemma}

\begin{proof} %Replacing $X$ by an open $\sigma$-compact subgroup that contains the compact set $K$, we can assume that the locally compact group $X$ is $\sigma$-compact. 
Let $\lambda$ be a Haar measure on $X$ and $F$ be the  set of points $x\in K$ such that $\lambda(O_x)>0$ for any neighborhood $O_x$ of $x$ in $K$. It is clear that $F$ is a closed subset of $K$. The inner regularity of the Haar measure $\lambda$ ensures that $\lambda(K\setminus F)=0$. We claim that the set $F$ is functionally closed in $X$.

Let $\mathcal H$ be the family of compact $G_\delta$-subgroups of $X$. Since $\mathcal H$ is closed under countable intersections, there exists a subgroup $H\in\mathcal H$ such that $\lambda(HF)=\inf\{\lambda(ZF):Z\in\mathcal H\}$. We claim that $HF=F$. Assuming that $HF\ne F$, we can find points $x\in H$ and $y\in F$ such that $xy\notin F$. Find a neighborhood $U\subseteq X$ of the identity of $X$ such that $Uxy\cap F=\emptyset$. By \cite[3.1.26]{AT}, the neighborhood $U$ contains a compact $G_\delta$-subgroup $Z$ of $X$. Replacing $Z$ by $Z\cap H$, we can assume that $Z\subseteq H$. It follows from $Zxy\cap F\subseteq Uxy\cap F=\emptyset$ that $xy\notin ZF$. Since the compact set $ZF$ is closed in $X$, there exists an open neighborhood $O_y$ of the point $y$ in $F$ such that $ZF\cap xO_y=\emptyset$.
The left-invariance of the Haar measure $\lambda$ and the definition of the set $F$ ensure that $\lambda(xO_y)=\lambda(O_y)>0$. Then $$\lambda(HF)\ge \lambda(xO_y)+\lambda(ZF)>\lambda(ZF),$$which contradicts the choice of the subgroup $H$. This contradiction shows that $HF=F$.

Let $X/H=\{Hx:x\in X\}$ be the quotient space and $q:X\to X/H$ be the quotient map. By the proof of Theorem 3.1.26 in \cite{AT}, the space $X/H$ is first-countable and by Vilenkin Theorem \cite[3.2.20]{AT}, the first-countable space $X/H$ is metrizable. Then the compact set $q[F]$ is functionally closed in $X/H$ and its preimage $F=HF=q^{-1}[q[F]]$ is functionally closed in $X$.
\end{proof}

\begin{lemma}\label{l:N-fBb} Let $\N$ be the $\sigma$-ideal of Haar-null sets in a $\sigma$-compact locally compact topological group $X$. Then
\begin{enumerate}
\item $\N$ has a functionally Borel base;
\item $\N$ is ccc;
\item $\Bo^\pm\!\N=\Ba^\pm\!\N$ is the $\sigma$-algebra of Haar-measurable sets in $X$.
\end{enumerate}
\end{lemma}

\begin{proof} Let $\lambda$ be a Haar measure on the locally compact group $X$. Since $X$ is $\sigma$-compact, $X=\bigcup_{n\in\w}K_n$ for some  increasing sequnce  $(K_n)_{n\in\w}$ of compact sets in $X$.
\smallskip

1. To show that the ideal $\N$ has a functionally Borel base, it suffices to check that every Haar-null Borel set $A\subseteq X$ is contained in a Haar-null functionally Borel subset of $X$.  For every  $n\in\w$, the Borel set $A\cap K_n$ has Haar measure zero and hence $\lambda(K_n\setminus A)=\lambda(K_n)<\infty$. By the inner regularity of the Haar measure, for every $m\in\w$ there exists a compact set $C_{n,m}\subseteq K_n\setminus A$ such that $\lambda(C_{n,m})<\lambda(K_n)-2^{-m}$. By Lemma~\ref{l:Hn-fB}, the compact set $C_{n,m}$ contains a functionally closed subset $F_{n,m}$ of $X$ such that $\lambda(F_{n,m})=\lambda(C_{n,m})$. It follows that $F=\bigcup_{n,m\in\w}F_{n,m}$ is a functionally Borel set in $X$ whose complement $X\setminus F$ contains $A$ and is Haar-null in $X$.
\smallskip

2. To prove that the ideal $\N$ is ccc, choose any disjoint family $\mathcal D$ of $\N$-positive Borel sets in $X$.  For every $n,m\in\w$ consider the family
$$\mathcal D_{n,m}=\{D\in\mathcal D:\lambda(D\cap K_n)\ge 2^{-m}\}.$$
The additivity of the measure $\lambda$ implies that the family $\mathcal D$ is finite and has cardinality $\le2^m\cdot \lambda(K_n)$. Now the $\sigma$-additivity of the measure $\lambda$ ensures that $\mathcal D=\bigcup_{n,n\in\w}\mathcal D_{n,m}$
and hence the family $\mathcal D$ is countable, witnessing that the ideal $\N$ is ccc.
\smallskip

3. By definition, $\Bo^\pm\!\N$ is the $\sigma$-ideal of Haar-measurable sets in $X$. To show that $\Bo^\pm\!\N=\Ba^\pm\!\N$, it suffices to check that every Borel subset $B$ of $X$ is $\Ba^\pm\!\N$-measurable. Let $\K$ be a maximal disjoint family of $\N$-positive compact subsets of $B$ and $\K'$ be a maximal disjoint family of $\N$-positive compact subsets of $X\setminus B$. The ccc property of the ideal $\N$ ensures that the families $\K$ and $\K'$ are countable. The inner regularity of the Haar measure ensures that the Borel sets $B\setminus\bigcup\K$ and $(X\setminus \bigcup\K')\setminus B$ are Haar-null and so is their union $X\setminus\bigcup(\K\cup\K')$. By Lemma~\ref{l:Hn-fB}, for every $K\in\K\cup\K'$ there exists a functionally closed set $F_K\subseteq K$ in $X$ such that $K\setminus F_K$ is Haar-null. Then $F=X\setminus \bigcup_{K\in\K\cup \K'}F_K$ is a functionally Borel subset of $X$ such that the symmetric difference $F\Delta B\subseteq \big(X\setminus\bigcup(\K\cup\K')\big)\cup\bigcup_{K\in\K\cup\K'}(K\setminus F_K)$ is Haar-null, witnessing that the Borel set $B$ is $\Ba^\pm\!\N$-measurable.
\end{proof}

The following lemma implies Theorem~\ref{t:main-H}.

\begin{lemma}\label{l:HM} A homomorphism $h:X\to Y$ from a locally compact topological group $X$ to a topological group $Y$ is continuous if and only if it is Haar-measurable.
\end{lemma}

\begin{proof} The ``only if'' part is trivial. To prove the ``if'' part, assume that the homomorphism $h$ is Haar-measurable. Let $\lambda$ be a Haar measure on $X$ and $H$ the subgroup of $X$ generated be any open compact neighborhood of the identity of $X$. Then $H$ is an open $\sigma$-compact subgroup of $X$. Since $\lambda$ restricted to the $\sigma$-algebra of Borel subsets of $H$ is a Haar measure on $H$, the Haar-measurability of $h$ implies the Haar-measurability of the restriction $h{\restriction}_H$. Being $\sigma$-compact, the locally compact group $H$ is $\w$-narrow and \v Cech complete. By Theorem~\ref{t:Cech}, $H$ is a Baire $K$-analytic group. 
By Lemma~\ref{l:N-fBb}, the $\sigma$-ideal $\N$ of Haar-null sets in $H$ has functionally Borel base, is ccc, and $\Ba^\pm\!\N=\Bo^\pm\!\N$ is the $\sigma$-algebra of Haar-measurable subsets of the locally compact group $H$. By Theorem~\ref{t:KA3}, the $\Bo^\pm\!\N$-measurable homomorphism $h{\restriction}_H$ is continuous and so is the homomorphism  $h$ (as $H$ is open in $X$).
\end{proof}

\section{Proof of Theorem~\ref{t:main-UM}}\label{s:9}

This section is devoted to the proof of Theorem~\ref{t:main-UM}. First we recall some definitions and known results.

By a {\em probability Radon measure} on a Tychonoff space $X$ we understand a $\sigma$-additive measure $\mu:\Bo(X)\to[0,\infty]$ on the $\sigma$-algebra of Borel subsets of $X$ such that $\mu(X)=1$ and for any Borel set $B\subseteq X$ and  real number $a<\mu(B)$ there exists a compact set $K\subseteq B$ of measure $\mu(K)>a$.

Let $\hat PX$ be the space of all probability Radon measure on $X$. The space $\hat PX$ carries the topology generated by the subbase consisting of the sets $\{\mu\in\hat PX:\mu(U)>a\}$ where $a\in\IR$ and $U$ runs over open subsets of $X$. 

Any continuous map $f:X\to Y$ between topological spaces induces a continuous map $\hat Pf:\hat PX\to\hat PY$ between the corresponding spaces of probability Radon measures. To each measure $\mu\in\hat PX$ the map $\hat Pf$ assigns the measure $\lambda\in \hat P Y$ such that $\lambda(B)=\mu(f^{-1}[B])$ for any Borel set $B$ in $Y$. The construction $\hat P$ determines a function in the category of Tychonoff spaces and their continuous maps. Categorial properties of this functor have been thoroughly studied in \cite{Ban95a}, \cite{Ban95b}. It is known that the $\hat P$ preserves perfect surjective maps between Tychonoff spaces.

A function $f:X\to Y$ between topological spaces is called {\em perfect} if it is  continuous, closed, and has compact preimages of points, see \cite[3.7]{Eng}. 

\begin{lemma}\label{l:sur} For any perfect surjective map $f:X\to Y$ between Tychonoff spaces, the map $\hat Pf:\hat PX\to \hat PY$ is perfect and surjective.
\end{lemma}

\begin{proof} By \cite[2.2]{Ban95a}, the map $\hat Pf$ is perfect and by \cite[2.9]{Ban95a}, the image $\hat Pf[\hat PX]$ is dense in $\hat PY$ and being closed, coincides with $\hat PY$. This means that the map $\hat Pf$ is surjective.
\end{proof}  

A subset $A$ of a topological space $X$ is called {\em universally measurable} if $A$ is {\em $\mu$-measurable} with respect to any probability Radon measure $\mu$ on $X$. The latter means that $\mu(B'\setminus B)=0$ for some Borel sets $B,B'$ in $X$ with $B\subseteq A\subseteq B'$.

\begin{lemma}\label{l:pum} Let $p:X\to Y$ be a perfect surjective map between  Tychonoff spaces. A subset $A\subseteq Y$ is universally measurable in $Y$ if its preimage $p^{-1}[A]$ is universally measurable in $X$.
\end{lemma}

\begin{proof} Let $A\subseteq Y$ be a subset whose preimage $p^{-1}[A]$ is universally measurable in $X$. To show that $A$ is universally measurable in $Y$, take any probability Radon measure $\mu$ on $Y$. By Lemma~\ref{l:sur}, there exists a probability Radon measure $\lambda$ on $X$ such that $\mu=\hat Pf(\lambda)$. Since the set $p^{-1}[A]$ is universally measurable, there exist Borel sets $B,B'$ in $X$ such that $B\subseteq p^{-1}[A]\subseteq B'$ and $\lambda(B'\setminus B)=0$. Since the measure $\lambda$ is Radon, there exist $\sigma$-compact sets $K\subseteq B$ and $K'\subseteq X\setminus B'$ such that $\lambda(K)=\lambda(B)$ and $\lambda(K')=\lambda(X\setminus B')$. Then $G=X\setminus K'$ is a $G_\delta$-set in $X$ such that $K\subseteq B\subseteq h^{-1}[A]\subseteq B'\subseteq G$ and $\lambda(G\setminus K)=0$. 

Consider the $\sigma$-compact set $S=p[K]$ and the $G_\delta$-set $S'=Y\setminus p[K']$ in $Y$. It follows from $K\subseteq p^{-1}[A] \subseteq X\setminus K'$ that $S\subseteq A\subseteq S'$ and $p^{-1}[S'\setminus S]\subseteq G\setminus K$. Then $\mu(S'\setminus S)=\lambda(p^{-1}[S'\setminus S])\le\lambda(G\setminus K)=0$, witnessing that the set $A$ is $\mu$-measurable and hence universally measurable.
\end{proof}

A function $f:X\to Y$ between topological spaces is {\em universally measurable} if for every open set $U\subseteq Y$ the preimage $f^{-1}[U]$ is universally measurable in $X$.

The following important fact was recently proved by Rosendal \cite{Ros19}.

\begin{lemma}[Rosendal]\label{l:Rosen} Every universally measurable homomorphism between Polish groups is continuous.
\end{lemma}

For a probability Radon measure $\mu$ on a topological space $X$ and a set $A\subseteq X$ let $$\mu^*(A)\defeq\inf\{\mu(B):A\subseteq B\in \Bo(X)\}$$ be {\em the outer $\mu$-measure} of $A$.

A universally measurable subset $A$ of a topological group group $X$ is called {\em left Haar-null} if there exists a probability Radon measure $\mu$ on $X$ such that  $\mu^*(xA)=0$ for all $x\in X$. 

The following lemma was proved by Rosendal \cite[2.8]{Ros09}. 

\begin{lemma}[Rosendal]\label{l:Ros} If a universally measurable set $A$ in a Polish group $X$ is not left Haar-null, then there exists a finite set $F\subseteq X$ such that $\bigcup_{x\in F}x^{-1}A^{-1}Ax$ is a neighborhood of the identity in $X$.
\end{lemma}

Also we shall need the following combinatorial lemma that can be found in \cite[3.31]{Pachl}.

\begin{lemma}[Pachl]\label{l:Pachl} Let $\kappa$ be an infinite cardinal, $A$ be a set of cardinality $<\kappa$ in a group $G$ and $\F$ be a finite family of subsets of $G$ such that $G=\bigcup_{F\in\F}FA$. Then there exists a set $F\in\F$ such that $G=F^{-1}FB$ for some set $B\subseteq G$ of cardinality $<\kappa$.
\end{lemma}

This lemma implies the following characterization of $\w$-narrowness, mentioned in \cite{BGR} and attributed there to Pachl.

\begin{lemma}[Pachl]\label{l:Pachl-w} A topological group $X$ is $\w$-narrow if and only if for any neighborhood $U$ of the identity in $X$ there exist a finite set $A\subseteq G$ and a countable set $B\subseteq G$ such that $G=AUB$.
\end{lemma}

\begin{proof} The ``only if'' part is trivial. To prove the ``if'' part, take any  neighborhood $U$ of the identity $e$ in $G$ and find a neighborhood $V$ of $e$ such that $V^{-1}V\subseteq U$. By our assumption, there exist a finite set $A$ and a countable set $B$ such that $G=AVB$. By Lemma~\ref{l:Pachl}, there exists an element $a\in A$ and a countable set $C$ in $G$ such that $G=(aV)^{-1}aVC=V^{-1}VC$ and hence $G=G^{-1}=C^{-1}V^{-1}V\subseteq C^{-1}U$, witnessing that the topological group $G$ is $\w$-narrow.
\end{proof}

\begin{lemma}\label{l:Haar} If $h:X\to Y$ is a universally measurable homomorphism from a Polish group $X$ to a topological group $Y$, then for every neighborhood $U$ of the identity in $Y$, the preimage $h^{-1}[U]$ is not left Haar-null in $X$.
\end{lemma}

\begin{proof} To derive a contradiction, assume that for some neighborhood $U$ of the identity in $Y$, the set $h^{-1}[U]$ is left Haar-null in $X$. Then there exists a probability Radon measure $\mu$ on $X$ such that $\mu^*(x\, h^{-1}[U])=0$ for all $x\in X$. By the definition of the outer measure $\mu^*$, the $\sigma$-ideal $\N_\mu=\{A\subseteq X:\mu^*(A)=0\}$ has a Borel base. Replacing the group $Y$ by $h[X]$, we can assume that the homomorphism $h$ is surjective.

Repeating the argument from the proof of Proposition~\ref{p:a}, find a $\sigma$-discrete open cover $\V$ of $Y$ such that every set $V\in\V$ is contained in a left shift $yU$ of the neighborhood $U$. Take any $x\in h^{-1}(y)$ and observe that
$$\mu^*(h^{-1}[V])\subseteq \mu^*(h^{-1}[yU])=\mu^{*}(xh^{-1}[U])=0.$$
Therefore, $\{h^{-1}[V]:V\in\V\}\subseteq\N_\mu$.

Write the $\sigma$-discrete family $\V$ as the union $\V=\bigcup_{n\in\w}\V_n$ of countably many discrete families $\V_n$. Since 
$$\textstyle 0<\mu(X)=\mu(h^{-1}[\bigcup\V])=\mu(h^{-1}[\bigcup_{n\in\w}\bigcup\V_n])=\mu(\bigcup_{n\in\w}h^{-1}[\bigcup\V_n])\le\sum\limits_{n\in\w}\mu^*(h^{-1}[\bigcup\V_n]),$$ for some $n\in\w$ the absolutely measurable set $h^{-1}[\bigcup\V_n]$ has nonzero outer measure and hence does not belong to the $\sigma$-ideal $\N_\mu$. By Theorem~\ref{t:4Poles}, the family $\J\defeq\{h^{-1}[V]:V\in\V_n\}\subseteq\N_\mu$ contains a subfamily $\J'$ whose union $\bigcup\J'$ is not $\Bo^\pm\!\N_\mu$-measurable. On the other hand, for the subfamily $\V_n'=\{V\in\V_n:h^{-1}[V]\in\J'\}$ and the open set $V'\defeq\bigcup\V_n'$, the preimage $h^{-1}[V']=\bigcup\J'$ is universally measurable and hence is $\Bo^\pm\N_\mu$-measurable. This contradiction completes the proof. 
\end{proof} 

\begin{lemma}\label{l:um=>w} If $h:X\to Y$ is a universally measurable homomorphism from a Polish group $X$ onto a topological group $Y$, then the topological group $Y=h[X]$ is $\w$-narrow.
\end{lemma}

\begin{proof} Assuming that $Y$ is not $\w$-narrow and applying Lemma~\ref{l:Pachl-w}, we can find an open neighborhood $U$ of the identity in $Y$ such that $G\ne FUC$ for any finite set $F$ and countable set $C$ in $Y$. By the continuity of group operations, there exists a neighborhood $V$ of the identity in $Y$ such that $V^{-1}V\subseteq U$. By Lemma~\ref{l:Haar}, the preimage $A\defeq h^{-1}[V]$ is not left Haar-null in $X$, and by Lemma~\ref{l:Ros}, there exists a finite set $F\subseteq X$ such that $W\defeq \bigcup_{x\in F}x^{-1}A^{-1}Ax$ is a neighborhood of the identity in the Polish group $X$. By the separability of $X$, there exists a countable set $C$ in $X$ such that $X=WC$. Then for the finite set $F'\defeq h[F]$ and countable set $C'\defeq h[FC]$ the equality $X=WC=FA^{-1}AFC$ implies $Y=F'V^{-1}VC'=F'UC'$, which contradicts the choice of $U$. This contradiction shows that $Y$ is $\w$-narrow. 
\end{proof}

The following lemma is due to Guran \cite{Guran} and can be found in \cite[3.4.19]{AT}.

\begin{lemma}\label{l:Guran} For any neighborhood $U$ of the identity in an $\w$-narrow topological group $X$, there exist a continuous homomorphism $h:X\to Y$ to a Polish group $Y$ and a neighborhood $V$ of the identity in $Y$ such that $h^{-1}[V]\subseteq U$.
\end{lemma}

\begin{lemma}\label{l:ump} Every universally measurable homomorphism $h:X\to Y$ from a Polish group $X$ onto a topological group $Y$ is continuous.
\end{lemma}

\begin{proof}  By Lemma~\ref{l:um=>w}, the topological group $Y=h[X]$ is $\w$-narrow. To check that $h$ is continuous, take any open neighborhood $U$ of the identity in $Y$. By Lemma~\ref{l:Guran}, there exist a continuous homomorphism $p:Y\to Z$ to a Polish group $Z$ and a neighborhood $V$ of the identity in $Z$ such that $p^{-1}[V]\subseteq U$. By Lemma~\ref{l:Rosen}, the universally measurable homomorphism $p\circ h:X\to Z$ is continuous and hence the sets $(p\circ h)^{-1}[V]=h^{-1}[p^{-1}[V]]\subseteq h^{-1}[U]$ are neighborhoods of the identity $e$ in $X$. This shows that $h$ is continuous at $e$ and hence is continuous everywhere.
\end{proof}

\begin{lemma}\label{l:umw=>w} Let $h:X\to Y$ be a universally measurable surjective homomorphism between topological groups. If $X$ is $\w$-narrow and \v Cech-complete, then the topological group $Y$ is $\w$-narrow.
\end{lemma}

\begin{proof} By Theorem~\ref{t:Cech}, the $\w$-narrow \v Cech-complete group $X$ contains a compact normal subgroup $K$ such that the quotient group $X/K$ is Polish. Let $q_X:X\to X/K$ be the quotient homomorphism. By \cite[1.5.7]{AT}, the map $q_X$ is perfect.

Since the homomorphism $h$ is universally measurable, the homomorphism $h{\restriction}_K:K\to Y$ is Haar-measurable. By Theorem~\ref{t:main-H}, the homomorphism $h{\restriction}_K$ is continuous and hence $H\defeq h[K]$ is a compact normal subgroup in $Y=h[X]$. Consider the quotient topological group $Y/H$ and the quotient homomorphism $q_Y:Y\to Y/H$. Since $\Ker(q_X)\subseteq \Ker(q_Y\circ h)$, there exists a unique homomorphism $\hbar:X/K\to Y/H$ such that $q_Y\circ\hbar=h\circ q_X$. 

We claim that $\hbar$ is universally measurable. Given any open set $U\subseteq Y/H$, consider its preimage $\hbar^{-1}[U]\subseteq X/K$ and observe that $q_X^{-1}[\hbar^{-1}[U]]=h^{-1}[q_Y[U]]$ is universally measurable by the universal measurability of $h$ and continuity of $q_Y$. By Lemma~\ref{l:pum}, the set $\hbar^{-1}[U]$ is universally measurable, witnessing that the homomorphism $\hbar:X/K\to Y/H$ is universally measurable. By Lemma~\ref{l:um=>w}, the topological group $Y/H$ is $\w$-narrow and by \cite[3.4.B]{AT}, the topological group $Y$ is $\w$-narrow.
\end{proof}

\begin{lemma}\label{l:umw=>c} Every universally measurable homomorphism $h:X\to Y$ from an $\w$-narrow \v Cech-complete topological group $X$ onto a topological group $Y$ is continuous.
\end{lemma}

\begin{proof}  By Lemma~\ref{l:umw=>w}, the topological group $Y=h[X]$ is $\w$-narrow. Given any neighborhood $U$ of the identity in $Y$, we should prove that $h^{-1}[U]$ is a neighborhood of the identity in $X$. By Lemma~\ref{l:Guran}, there exist a continuous homomorphism $p:Y\to Z$ to a Polish group $Z$ and an open neighborhood $V$ of the identity $e$ in $Z$ such that $p^{-1}[V]\subseteq U$. 

By Theorem~\ref{t:Cech}, the $\w$-narrow \v Cech-complete group $X$ contains a compact normal $G_\delta$-subgroup $H$ such that the quotient group $G/H$ is Polish. The universal measurability of the homomorphism $h$ implies the Haar-measurability of the  restriction $h{\restriction}_H$. By Theorem~\ref{t:main-H}, the homomorphism $h{\restriction}_H$ is continuous and so is the homomorphism $p\circ h{\restriction}_H:H\to Z$. Then its kernel $K$ is a compact $G_\delta$-subgroup in $H$ and also in $X$ (as $H$ is of type $G_\delta$ in $X$).

By  \cite[4.3.26]{AT}, the quotient group $X/K$ is $\w$-narrow and \v Cech-complete. Since $K$ is a $G_\delta$-group in $X$, the quotient group $X/K$ has countable pseudocharacter and by Theorem~\ref{t:Cech}, $X/K$ is Polish. Let $q:X\to X/K$ be the quotient homomorphism. Since $K\subseteq (p\circ h)^{-1}(e)$, there exists a unique homomorphism $\hbar:X/K\to Z$ such that $p\circ h=\hbar\circ q$. Repeating the argument from the proof of Lemma~\ref{l:umw=>w}, we can show that the univeral measurability of $h$ implies the universal measurability of the homomorphism $\hbar:X/K\to Z$. By Lemma~\ref{l:ump}, the homomorphism $\hbar:X/K\to Z$ is continuous and hence the sets $q^{-1}[\hbar^{-1}[V]]=h^{-1}[p^{-1}[V]]\subseteq h^{-1}[U]$ are neighborhoods of the identity in $X$.  This shows that the homomorphism $h$ is continuous at the identity of $X$ and hence is continuous everywhere.
\end{proof}

Our final lemma implies Theorem~\ref{t:main-UM}.

\begin{lemma} A homomorphism $h:X\to Y$ from a \v Cech-complete topological group $X$ to a topological group $Y$ is continuous if and only if it is universally measurable.
\end{lemma}

\begin{proof} The ``only if'' part is trivial. To prove the ``if'' part, assume that the homomorphism $h$ is universally measurable. Since \v Cech-complete spaces are $k$-spaces \cite[3.9.5]{Eng}, the continuity of $h$ will follow as soon as we check that for every compact subset $K\subseteq X$ the restriction $h{\restriction}_K$ is continuous. Given any compact subset $K\subseteq X$, consider the $\sigma$-compact subgroup $H$ of $X$ generated by the compact set $K$. By \cite[3.4.6]{AT}, the $\sigma$-compact group $H$ is $\w$-narrow and so is its closure $\bar H$ in $X$, see \cite[3.4.9]{AT}. Since closed subspaces of \v Cech-complete spaces are \v Cech-complete \cite[3.9.6]{Eng}, the topological group $\bar H$ is \v Cech-complete. The universal measurability of the homomorphism $h$ implies the universal measurability of the restriction $h{\restriction}_{\bar H}:\bar H\to Y$. By Lemma~\ref{l:umw=>c}, the homomorphism $h{\restriction}_{\bar H}$ is continuous and so is the restriction $h{\restriction}_K$.
\end{proof}

\section{Acknowledgement} The author expresses his sincere thanks to Robert Ra\l owski and Szymon \.Zeberski for valuable discussions and the idea of attacking the problem of Kuznetsova applying the results of their joint  paper \cite{BRZ}. 

%\newpage


\begin{thebibliography}{}

\bibitem{Alfsen} E.M.~Alfsen, {\em A simplified constructive proof of the existence and uniqueness of Haar measure}, Math. Scand. {\bf 12} (1963), 106--116. 


\bibitem{AT} A.V.~Arhangelskii, M.~Tkachenko, {\em Topological groups and related structures}, Atlantis Press / World Sci., 2008.

%\bibitem{BBC}  M.~Balcerzak, A.~Bartoszewicz, K.~Ciesielski,  {\em On Marczewski--Burstin representations of certain algebras}, Real Analysis Exchange (2000), 581--592. 

%\bibitem{BBK} M.~Balcerzak, A.~Bartoszewicz, P.~Koszmider, {\em On Marczewski--Burstin representable algebras}, Colloq. Math. {\bf 99} (2004), 55--60.

\bibitem{Banach} S.~Banach, {\em Sur l'\'equation fonctionnelle $f(x+y)=f(x)+f(y)$}, Fund. Math. {\bf 1} (1920), 123--124.

\bibitem{Ban95a} T.~Banakh, {\em Topology of spaces of probability measures. I. The functors $P_\tau$ and $\hat P$}, Mat. Stud. {\bf 5} (1995), 65--87; (Engl. transl. {\tt arxiv.org/abs/1112.6161}).


\bibitem{Ban95b} T.~Banakh, {\em  Topology of spaces of probability measures. II. Barycenters of Radon probability measures and the metrization of the functors $P_\tau$ and $\hat P$}, Mat. Stud. {\bf 5} (1995), 88--106; (Engl. transl. {\tt arxiv.org/abs/1206.1727}). 

\bibitem{BJGS} T.~Banakh, S.~Gl\c ab, E.~Jab\l o\'nska, J.~Swaczyna, {\em Haar-$\mathcal I$ sets: looking at small sets in Polish groups through compact glasses}, Dissert.~Math. {\bf 564} (2021), 105 pp. 

\bibitem{BGR} T.~Banakh, I.~Guran, A.~Ravsky, {\em Generalizing separability, precompactness and narrowness in topological groups}, Rev. R. Acad. Cienc. Exactas Fís. Nat. Ser. A Mat. RACSAM {\bf 115}:1 (2021),  Paper No. 18, 7 pp.

\bibitem{BRZ} T.~Banakh, R.Ra\l owski, Sz.~\. Zeberski, {\em The Set-Cover game and nonmeasurable unions}, preprint\newline ({\tt arxiv.org/abs/2011.11342}).

\bibitem{BR} T.~Banakh, A.~Ravsky, {\em Banalytic spaces and characterization of Polish groups}, preprint\newline ({\tt arxiv.org/pdf/1901.10732.pdf}).%Letters in Math. Sci. (published online at {\tt www.prior-sci-pub.com/lims\_2019art1.html}).

%\bibitem{Brzdek} J.~Brzd\c ek, {\em Continuity of measurable homomorphisms}, Bull. Aust. Math. Soc. 78 (2008), no. 1, 171--176.

\bibitem{BO1} N.~Bingham, A.~Ostaszewski, {\em The Steinhaus-Weil property: I. Subcontinuity and amenability}, Sarajevo J. Math. {\bf 16}(29) (2020), 13--32.

\bibitem{BO2} N.~Bingham, A.~Ostaszewskii, {\em The Steinhaus-Weil property: II. The Simmons-Mospan converse}, Sarajevo J. Math. {\bf 16}(29) (2020), 179--186.

\bibitem{BO3} N.~Bingham, A.~Ostaszewski, {\em The Steinhaus-Weil property: III. Weil topologies}, Sarajevo J. Math. {\bf 17}(30) (2021) 129--142. 

\bibitem{BCGRN} J.~Brzuchowski, J.~Cicho\'n, E.~Grzegorek, C.~Ryll-Nardzewski, {\em On the existence of nonmeasurable unions}, Bull. Acad. Polon. Sci. Math. {\bf  27}:6 (1997),  447--448.

\bibitem{Cauchy} A.L.~Cauchy, {\em Cours d'Analyse de l'Ecole Royale Polytechnique}. Chez Debure fr\`eres, 1821.

%\bibitem{CK} J.~Cicho\'n, A.~Kharazishvili, {\em On ideals with projective base},  Georgian Mathematical Journal, {\bf 9}:3 (2002)  461--472.

\bibitem{Chris71} J.P.R.~Christensen, {\em Borel structures in groups and semigroups}, Math. Scand. {\bf 28} (1971), 124--128.

\bibitem{Chris} J.P.R.~Christensen, {\em Topology and Borel structure}, North-Holland Publishing Co., Amsterdam-London; 1974. 

%J. P. R. Christensen, Borel structures in groups and semigroups, Math. Scand. 28 (1971), 124–128.

\bibitem{Chris72} J.~Christensen, {\em On sets of Haar measure zero in abelian Polish groups}, Israel J. Math. {\bf 13} (1972), 255--260.

\bibitem{EN} M.~Elekes, D.~Nagy, {\em Haar null and Haar meager sets: a survey and new results}, Bull. Lond. Math. Soc. {\bf 52}:4 (2020), 561--619.

%\bibitem{E} E.~Ellentuck, {\em A new proof that analytic sets are Ramsey}, Journal of Symbolic Logic {\bf 39} (1974), 163--165.

\bibitem{Eng} R.~Engelking, {\em General Topology}, Heldermann Verlag, 1989.

%\bibitem{Fab} M.~Fabian, {\em G\^ateaux differentiability of convex functions and topology}, Wiley-Intersci. Publ., NY, 1997.

\bibitem{Frechet} M.~Fr\'echet, {\em Pri la funckcia ekvacio $f(x+y)=f(x)+f(y)$}, L'Ensenement Mathematique, {\bf 15} (1913), 390--393.

\bibitem{Fremlin} D.~Fremlin, {\em Measure Theory}, Vol.4, Torres Fremlin, Colchester, 2006.


%\bibitem{GSS} M.~Goldstern, M.~Repick\'y, S.~Shelah, O.~Spinas, {\em On tree ideals}, Proc. Amer. Math. Soc. {\bf 123}:5 (1995), 1573-1581.

%{\color{blue}\bibitem{Guran} I.I.~Guran, {\em Topological groups similar to Lindel\"of groups}, Dokl. Akad. Nauk SSSR {\bf 256}:6 (1981),  1305--1307.

%\bibitem{KKP} J.~K\c akol, W.~Kubi\'s, M.~L\'opez-Pellicer, {\em Descriptive topology in selected topics of functional analysis}, Springer, New York, 2011. }



%\bibitem{Marczewski} E.~Marczewski (Szpilrajn), {\em Sur une classe de fonctions de W. Sierpi\'nski et la classe correspondante d’ensembles}, Fund. Math. {\bf 24} (1935), 17--34.

%\bibitem{JMSS} H.~Judah, A.~Miller, S.~Shelah, {\em Sacks forcing, Laver forcing and Martin’s Axiom}, Archive for Math Logic {\bf 31} (1992) 145--161.

\bibitem{Guran} I.J.~Guran, {\em Topological groups similar to Lindel\"of groups},  Dokl. Akad. Nauk SSSR. {\bf 256}:6 (1981), 1305--1307.

\bibitem{Hamel} G.~Hamel, {\em Eine Basis aller Zahlen und die unstetigen Lösungen der Funktionalgleichung $f(x+y)=f(x)+f(y)$}, Math. Ann. {\bf 60} (1905), 459--462.

\bibitem{Jab} E.~Jab\l o\'nska, {\em Some analogies between Haar meager sets and Haar null sets in abelian Polish groups}, J. Math. Anal. Appl. {\bf 421}:2 (2015),  1479--1486.

\bibitem{Jech} T.~Jech, {\em Set Theory}, Springer-Verlag, Berlin, 2003.

%\bibitem{JW} W.~Just, M.~Weese, {\em Discovering Modern Set Theory, I}, GSM, {\bf 8}, Amer. Math. Soc., Providence, RI, 1996.

%\bibitem{LR} G.~\L ab\c{e}dzki, M.~Repick\'y, {\em Hechler reals}, J. Symb. Log. {\bf 60}:2 (1995),  444--458.

\bibitem{Ke} A.~Kechris, {\em Classical Descriptive Set Theory}, Springer, 1995.

\bibitem{Klep1} A.~Kleppner, {\em Measurable homomorphisms of locally compact groups}, Proc. Amer. Math. Soc. {\bf 106}:2 (1989),  391--395.

\bibitem{Klep2} A.~Kleppner, {\em Correction to: "Measurable homomorphisms of locally compact groups'' [Proc. Amer. Math. Soc. 106 (1989), no. 2, 391–395]}, Proc. Amer. Math. Soc. 111 (1991), no. 4, 1199--1200.

\bibitem{Kuczma} M.~Kuczma, {\em An introduction to the theory of functional equations and inequalities}, Universytet \'Sl\c aski, Warszawa--Krak\'ow--Katowice, 1985.

\bibitem{Kuz} Y.~Kuznetsova, {\em On continuity of measurable group representations and homomorphisms}, Studia Math. {\bf 210}:3 (2012),  197--208.


%\bibitem{Matet} P.~Matet, {\em A short proof of Ellentuck's theorem}, Proc. Amer. Math. Sci. {\bf 129}:4 (2000), 1195--1197. 

%\bibitem{Prikry} G.~Koumoullis, K.~Prikry, {\em Perfect measurable spaces}, Annals of Pure and Applied Logic, {\bf 30} (1986), 219--248. 

%\bibitem{Ke} A.~Kechris, {\em Classical Descriptive Set Theory}, Springer, 1995.

\bibitem{Pachl} J.~Pachl, {\em Uniform spaces and measures}. Fields Institute Monographs, 30. Springer, New York; Fields Institute for Research in Mathematical Sciences, Toronto, ON, 2013. x+209 pp.

\bibitem{Pettis} B.J. Pettis, {\em Remarks on a theorem of E.J. McShane}, Proc. Amer. Math. Soc. {\bf  2} (1951), 166--171.

\bibitem{Ros09} C.~Rosendal, {\em Automatic continuity of group homomorphisms}, Bull. Symbolic Logic {\bf 15}:2 (2009), 184--214.

\bibitem{Ros19} C.~Rosendal, {\em Continuity of universally measurable homomorphisms}, Forum Math. Pi. {\bf 7} (2019), e5, 20 pp.

%\bibitem{Repicki} M.~Repick\'y, {\em Perfect sets and collapsing continuum}, Comment. Math. Univ. Carolin. {\bf 44}:2 (2003) 315--327.

\bibitem{RJ} C.A.~Rogers, J.E.~Jayne, {\em K-analytic sets}, in: {\em Analytic Sets},  Academic Press, (1980), 1--181.

\bibitem{Sierpinski} W.~Sierpi\'nski, {\em Sur l’\'equation fonctionnelle $f(x+y)=f(x)+f(y)$}, Fund. Math. {\bf 1} (1920), 116--122.

\bibitem{Steinhaus} H.~Steinhaus, {\em Sur les distances des points dans les ensembles de mesure positive}, Fund. Math. {\bf 1} (1920), 93--104.

\bibitem{Strom} K.~Stromberg, {\em An elementary proof of Steinhaus's theorem}, Proc. Amer. Math. Soc. {\bf 36} (1972), 308.


\end{thebibliography}
\end{document}